\documentclass{amsart}
\usepackage{amsmath,amsfonts,amssymb,amsthm,hyperref}
\usepackage{comment}
\usepackage{color}

\newtheorem{lem}{Lemma}
\newtheorem{thm}[lem]{Theorem}
\newtheorem{prop}[lem]{Proposition}
\newtheorem{cor}[lem]{Corollary}
\newtheorem{conj}[lem]{Conjecture}
\theoremstyle{definition}

\theoremstyle{remark}

\DeclareMathOperator{\vol}{vol}

\newcommand{\R}{\mathbb{R}}

\usepackage[isbn=false, url=false, doi=false, eprint=false]{biblatex}
\title{The equality case in Godbersen's inequality and its generalizations}
\author{Omer Friedland, Eli Putterman and Maud Szusterman}

\begin{document}

\begin{abstract} Godbersen's long-standing conjecture states that the mixed volumes of any convex body $K \subset \R^n$ satisfy $V(K[k], -K[n-k])\le \binom{n}{k} \vol(K)$ for any $k \in \{1, \dots, n - 1\}$, with equality if and only if $K$ is an $n$-simplex. Kotrbat\'y and Mouamine recently proved the inequality in full generality, and showed that if $K$ is a (full-dimensional) polytope attaining equality, it must be a simplex. We complete the characterization of the equality case: if $K \subset \R^n$ is a convex body for which $V(K[k], -K[n-k]) = \binom{n}{k} \vol(K)$ for some $k \in \{1, \dots, n -1\}$, then $K$ is an $n$-simplex. By a similar method, we characterize the equality case in the higher-order Godbersen inequality conjectured by Schneider and by Kotrbat\'y.
\end{abstract}

\maketitle

\section{Introduction}
We write $\mathcal K^n$ for the set of compact convex sets in $\R^n$, $\vol_n(K)$ for the volume of a convex set $K \in \mathcal K^n$, and say that a compact convex set is a convex body if $\vol_n(K) > 0$. We use the notation $V_n(K_1, \dots, K_n)$ for the mixed volume of $n$ convex bodies in $\R^n$, and in the arguments to $V_n$, we use $K[j]$ to denote $j$ copies of $K$.

In 1938, Godbersen \cite{godbersen1938satz} made the following conjecture:

\begin{conj}[Godbersen's conjecture]\label{conj:god}
For any convex body $K\subset \R^n$ and any $1\le k\le n-1$,  
\begin{equation*}
V_n(K[k], -K[n-k])\le \binom{n}{k} \vol_n(K),\end{equation*}
with equality if and only if $K$ is an $n$-simplex.
\end{conj}

After partial results by several authors \cite{AEFO2015,ASS2023,sadovsky2024,AP2025}, the inequality part of the conjecture was proven recently by Kotrbat\'y-Mouamine \cite{KM2026}, who also obtained the equality case under the assumption that $K$ is a polytope:

\begin{thm}[{{\cite[Theorem 1.1]{KM2026}}}] \label{thm:km_ineq}
For any convex body $K\subset \R^n$ and any $1\le k\le n-1$,  
\begin{equation}\label{eq:Godbersen-ineq}
V_n(K[k], -K[n-k])\le \binom{n}{k} \vol_n(K).\end{equation}
If $K$ is an $n$-polytope and equality holds in \eqref{eq:Godbersen-ineq}, then $K$ is an $n$-simplex.
\end{thm}

In the first main result of this note, we characterize the equality case of \eqref{eq:Godbersen-ineq} among all convex bodies in $\R^n$, completing the resolution of Godbersen's conjecture:

\begin{thm}\label{thm:godbersen_equality} If $K \subset \R^n$ is a convex body satisfying $V_n(K[k], -K[n - k]) = \binom{n}{k} \vol_n(K)$ for some $k\in \{1,\dots,n-1\}$, then $K$ is an $n$-simplex.
\end{thm}

The proof of the equality case for polytopes in \cite{KM2026} relies heavily on the face structure of polytopes and on known properties of the support of their mixed area measures. We adopt a different approach, relating the equality case in Godbersen's conjecture to the recent higher-order generalization of Godbersen's conjecture. 

To state this generalization requires a bit more notation. For a positive integer $p$, we think of $\R^{pn}$ as the Cartesian product $\overbrace{\R^n \times \cdots \times \R^n}^{\text{$p$ times}}$. Let $\imath_1, \dots, \imath_p: \R^n \to \R^{pn}$ be the $p$ block coordinate embeddings, and let $\Delta_p: \R^n \to \R^{pn}$, $\Delta(x) = (x, x, \dots, x)$ be the diagonal embedding.

\begin{conj}[Higher-order Godbersen's conjecture]\label{conj:higher_godbersen}
For any convex body $K\subset \R^n$, any integer $p \ge 1$, and any $a_0, \dots, a_p \in \{1, \dots, n\}$ with $a_0 + \cdots + a_p = pn$, one has
\begin{equation}\label{eq:higher_godbersen}
V_{pn}(-\Delta_p K[a_0], \imath_1 K[a_1], \dots, \imath_p K[a_p]) \le \binom{n}{n - a_0, \dots, n - a_p} \frac{(n!)^p}{(pn)!} \vol_n(K)^p,\end{equation}
with equality if and only if $K$ is an $n$-simplex.
\end{conj}

Conjecture \ref{conj:higher_godbersen} was put forward by Schneider \cite{S2000}, without an explicit computation of the right-hand side of \eqref{eq:higher_godbersen} for a simplex, and later by Kotrbat\'y \cite{K2025}, who proved it in the class of anti-blocking bodies. Kotrbat\'y \cite{K2025} also drew interesting connections between Conjecture \ref{conj:higher_godbersen} and valuation theory. Fry\v{s} and Kotrbat\'y \cite{FK2026} have also recently studied Conjecture \ref{conj:higher_godbersen} extensively, proving (among other results) higher-order generalizations of the results of \cite{AEFO2015,AP2025}.

It was known to the experts, and is stated in the introduction of \cite{KM2026}, that Kotrbat\'y-Mouamine's proof of \eqref{eq:Godbersen-ineq} extends immediately to yield a proof of the ``higher-order'' inequality \eqref{eq:higher_godbersen}. As this argument has yet to appear in the literature (although a certain special case recently appeared in \cite{CFLXX2026}), we give it below, in \S \ref{sec:higher_godbersen}, in full detail; we claim no originality for this proof.

In our second main result, we extend the characterization of the equality case to the Godbersen-type inequalities of any order:

\begin{thm}\label{thm:higher_godbersen_equality} If $K \subset \R^n$ is a convex body satisfying equality in \eqref{eq:higher_godbersen} for some $p\ge 1$ and some $a_0, \ldots, a_p \in \{1, \ldots, n\}$ with $a_0 + \cdots + a_p = pn$, then $K$ is an $n$-simplex.
\end{thm}

The proof of Theorem \ref{thm:godbersen_equality} goes as follows. As shown in \cite{KM2026}, equality in Godbersen's inequality for a given $K$ implies equalities between certain mixed volumes of different convex sets in $\mathbb R^{2n}$ constructed from $K$ (cf. \eqref{eq:godbersen_equality_lifted}), which in particular give equality cases in the corresponding Alexandrov-Fenchel inequalities. Applying a standard characterization of the equality case in Alexandrov-Fenchel yields equality between two families of mixed volumes in $\R^{2n}$ with $s\leq n - 1$ free convex body parameters (cf. \eqref{eq:godbersen_free}). By tuning $s\in \{1, \ldots,n-1\}$ and by choosing the free parameters carefully, we obtain that certain mixed volumes (denoted $H_{a,b,c}$ in \S 4, and which correspond to the LHS in \eqref{eq:higher_godbersen}) must satisfy equality in \eqref{eq:higher_godbersen}. Using these equalities, we then show that $K$ must also attain equality in the (long-known) $k = 1$ case of Godbersen's equality. Since equality in the case $k = 1$ of Godbersen's inequality is known to characterize the simplex, this suffices. 

The proof of Theorem \ref{thm:higher_godbersen_equality} goes along the same lines, but is a bit more complicated: it requires an additional application of the method of proof of the higher-order Godbersen's inequality \eqref{eq:higher_godbersen} (as presented in section \ref{sec:higher_godbersen}), applied to a different inclusion (see \eqref{eq:inclusionp0} below). Although Theorem \ref{thm:godbersen_equality} is a special case of Theorem \ref{thm:higher_godbersen_equality}, we include a separate proof of the former, as it is simpler and already contains all the ideas needed for proving Theorem \ref{thm:higher_godbersen_equality}.

The remainder of this note is organized as follows. In \S \ref{sec:preliminaries} we state the background results we shall use. In \S \ref{sec:higher_godbersen} we prove the higher-order Godbersen's inequality, \eqref{eq:higher_godbersen}. In \S \ref{sec:godbersen_equality} we prove Theorem \ref{thm:godbersen_equality}, the equality case in Godbersen's inequality; and in \S \ref{sec:higher_godbersen_equality} we prove Theorem \ref{thm:higher_godbersen_equality}, the equality case in the higher-order Godbersen inequality.

\subsection*{Related work} 
In \cite{K2025}, Kotrbat\'y also made the following conjecture, which is weaker than Conjecture \ref{conj:higher_godbersen}: for every convex body $K \in \mathcal K^n$,
\begin{equation}\label{eq:weak_higher_godbersen} V_{pn}(-\Delta_p K[k], K^p[pn - k]) \le \binom{n}{k} \vol(K)^p 
\end{equation}
where $K^p = \overbrace{K \times \cdots \times K}^{\text{$p$ times}}$, with equality if and only if $K$ is a simplex. Recently, Fry\v{s} \cite{F2026} found a proof of \eqref{eq:weak_higher_godbersen} for $p > 1$ which also demonstrates that equality in \eqref{eq:weak_higher_godbersen} holds only if $K$ is a simplex.

Concurrently with this work, Fry\v{s},  Kotrbat\'y, Mouamine and Wang \cite{FKMW2026}  developed the geometric method of \cite{KM2026} further and were able to use it to characterize the equality cases in the Godbersen inequalities of any order by a completely different route. In the first version of this note, we only proved Theorem \ref{thm:godbersen_equality}; shortly after the appearance of \cite{FKMW2026}, we figured out how to push our method further in order to prove Theorem \ref{thm:higher_godbersen_equality} as well.

\subsection*{Acknowledgements} 
We would like to thank Filip Fry\v{s} for showing us his proof of the ``two-body Godbersen inequality'' (not yet published), 
from which we drew inspiration for the present
proof of Theorem \ref{thm:godbersen_equality}; Mohamed Mouamine for showing us Proposition \ref{prop:moreequalities-AFargument} and its proof; and the organizers of the Prague Conference in Convex Geometry, where this work began.

\section{Preliminaries}\label{sec:preliminaries}
Our general reference for convex geometry is the book of Schneider \cite{schneider2014book}; any results stated below without further reference may be found there.

We recall some basic properties of mixed volumes: for any $K, L, K_1, \dots, K_n \in \mathcal K^n$, we have:

\begin{enumerate}
\item[(i)] \emph{(non-negativity)} $V_n(K_1, \dots, K_n) \ge 0$.
\item[(ii)] $V_n(K[n]) = \vol_n(K)$.
\item[(iii)] \emph{(multilinearity)} For $a, b > 0$, 
\[V_n(aK + bL, K_2, \dots, K_n) = a V_n(K, K_2, \dots, K_n) + b V_n(L, K_2, \dots, K_n).\]
\item[(iv)] \emph{(translation invariance)} For any $x_1, \dots, x_n \in \R^n$, $V_n(K_1 + x_1, \dots, K_n + x_n) = V_n(K_1, \dots, K_n)$.
\item[(v)] For any $d \le n$, if $K_1, \ldots, K_d$ all lie in an affine subspace of dimension less than $d$, then
\[V_n(K_1, \ldots, K_d, K_{d + 1}, \dots, K_n) = 0.\]
\item[(vi)] \emph{(permutation invariance)} For any permutation $\sigma \in S_n$, 
\[V_n(K_{\sigma(1)}, \dots, K_{\sigma(n)}) = V_n(K_1, \dots, K_n).\]
\item[(vii)] \emph{(monotonicity)} If $K \subset L$ then 
$V_n(K, K_2, \dots, K_n) \le V_n(L, K_2, \dots, K_n)$.
\item[(viii)] \emph{(linear invariance)} If $A: \R^n \to \R^n$ is a linear transformation then
\[V_n(AK_1, AK_2, \dots, AK_n) = |\det A| V_n(K_1, K_2, \dots, K_n).\]
\end{enumerate}

For $K \in \mathcal K^n$, let 
\[q(K) = \inf \{t > 0\,|\, \exists x \in \R^n: -K + x \subset t K\},\]
which is called the Minkowski measure of symmetry. By monotonicity and translation-invariance of mixed volumes, one has 
\[V_n(-K, K[n - 1]) \le V_n(q(K) K, K[n - 1]) = q(K) \vol_n(K).\]
Moreover, it is known that $q(K) \le n$ for any $K \in \mathcal K^n$, with equality only for the simplex \cite{schneider2009stability}; in particular:

\begin{prop}\label{prop:godbersen_j=1_equality} Equality in the case $k = 1$ of Godbersen's inequality \eqref{eq:Godbersen-ineq} holds only for a simplex.
\end{prop}

As in the proof of Godbersen's inequalities, the main tool in the proof of their higher-order generalization \eqref{eq:higher_godbersen} is the projection formula for mixed volumes:

\begin{lem}[{\cite[Theorem 5.3.1]{schneider2014book}}]
\label{lem:exact}
Let $1\leq d\leq N-1$, let $E\subset\R^N$ be a $d$-dimensional linear subspace, let $P_{E^\perp}$ denote the orthogonal projection onto $E^{\perp}$, and let $L_1,\dots,L_d,M_1,\dots,M_{N-d}\in \mathcal K^N$ with $L_i\subset E$ for each $i=1,\dots,d$. Then
\begin{equation}
\label{eq:projectionidentity}
\binom{N}{d} V_N(L_1,\dots,L_d,M_1,\dots,M_{N-d}) = V_E(L_1,\dots,L_d) V_{E^\perp}(P_{E^\perp} M_1,\dots,P_{E^\perp} M_{N-d}),
\end{equation}
where $V_E$ and $V_{E^\perp}$ denote the mixed volumes in $E$, $E^\perp$, respectively.
\end{lem}

We shall also need the Alexandrov-Fenchel inequality, with its equality case \cite[Theorem 7.4.2]{schneider2014book}:

\begin{thm}[Alexandrov-Fenchel inequality] \label{thm:af_equality}
Let $K,L,C_1,\dots,C_{n-2} \in \mathcal K^n$ be convex sets with $V_n(K,L,C_1,\dots,C_{n-2})>0$.
Then 
\[V_n(K,L,C_1,\dots,C_{n-2})^2 \ge 
V_n(K,K,C_1,\dots,C_{n-2})\,
V_n(L,L,C_1,\dots,C_{n-2}),\]
and moreover, the following are equivalent:
\begin{enumerate}
\item  $V_n(K,L,C_1,\dots,C_{n-2})^2=
V_n(K,K,C_1,\dots,C_{n-2})\,
V_n(L,L,C_1,\dots,C_{n-2})$.
\item There exists $\alpha > 0$ such that for any $M \in \mathcal K^n$, 
\[V_n(M,K,C_1,\dots,C_{n-2}) = \alpha V_n(M,L,C_1,\dots,C_{n-2}).\]
\end{enumerate}
\end{thm}

\section{Proof of the higher-order Godbersen inequality}\label{sec:higher_godbersen}

Our goal in this section is to prove the following theorem. Recall that $\imath_1, \dots, \imath_p: \R^n \to \R^{pn}$ and $\Delta_p: \R^n \to \R^{pn}$ denote, respectively, the $p$ block coordinate embeddings and the diagonal embedding.

\begin{thm}\label{thm:higher_godbersen} For any $K \in \mathcal K^n$, one has \begin{equation}\label{eq:higher_godbersen_redux}
V_{pn}(-\Delta_p K[a_0], \imath_1 K[a_1], \dots, \imath_p K[a_p]) \le \binom{n}{n - a_0, \dots, n - a_p}  \frac{(n!)^p}{(pn)!} \vol_n(K)^p.
\end{equation}
\end{thm}

\begin{proof} First of all, if $\vol_n(K) = 0$, the affine span of $K$ is not all of $\R^n$. In this case, after translation, all the bodies $-\Delta_p K, \imath_1 K, \dots, \imath_p K$ will lie in a proper affine subspace of $\R^{pn}$, so the left-hand side of \eqref{eq:higher_godbersen_redux} is $0$ as well. Thus we may assume that $\vol_n(K) > 0$.

As in \cite{KM2026}, we lift to a space with one more coordinate block. Let \[\tilde\imath_0, \dots, \tilde\imath_{p}: \R^n \to \R^{(p + 1) n}\]
be the $p + 1$ block coordinate embeddings of $\R^n$ into $\R^{(p + 1) n}$. (Note that in $\mathbb R^{(p + 1) n}$, we choose to index the coordinate blocks starting from $0$.) Also let $\Delta_{p + 1}: \R^n \to \R^{(p + 1) n}$ be the diagonal embedding. We consider the mixed volumes
\begin{equation}\label{eq:lifted_mv}
    \mathcal V_{a_0, \dots, a_p} := V_{(p + 1)n}(\Delta_{p + 1} K[n], \tilde\imath_0 K[a_0], \dots, \tilde\imath_p K[a_p]).
\end{equation}

Let $A \in \R^{(p + 1)n \times (p + 1) n}$ be the block matrix
\[A = \begin{pmatrix}
-I_n   & \\ 
-I_n   & I_n & \\
\vdots &     & \ddots & \\
-I_n   &     &        &  I_n
\end{pmatrix},\]
which has determinant $(-1)^n$ and satisfies
\begin{align*}
    A\Delta_{p + 1} &= -\tilde\imath_0 \\
    A\tilde\imath_0 &= -\Delta_{p + 1} \\
    A\tilde\imath_k &= \tilde\imath_k, \qquad k = 1, \dots, p.
\end{align*}

Then by linear invariance of mixed volumes (viii), we have
\begin{align*}
\mathcal V_{a_0, \dots, a_p} &= V_{(p + 1)n}(A\Delta_{p + 1} K[n], A\tilde\imath_0 K[a_0], A\tilde\imath_1 K[a_1], \dots, A\tilde\imath_p K[a_p]) \\
&= V_{(p + 1)n}(-\tilde\imath_0 K[n], -\Delta_{p + 1} K[a_0], \tilde\imath_1 K[a_1], \dots, \tilde\imath_p K[a_p]).
\end{align*}

Let $E = \operatorname{Im} \tilde\imath_0$ be the span of the first coordinate block in $\R^{(p + 1)n}$; then $-\tilde\imath_0 K \subset E$ and $\dim(E)=n$, so by Lemma \ref{lem:exact}, we have
\begin{equation}\label{eq:proa_pplus1}
\mathcal V_{a_0, \dots, a_p} = \binom{(p + 1) n}{n}^{-1} \vol_n(K) \cdot V_{E^\perp}(P_{E^\perp}(-\Delta_{p + 1} K)[a_0], P_{E^\perp} \tilde\imath_1 K[a_1], \dots, P_{E^\perp} \tilde\imath_p K[a_p]).
\end{equation}

Note that $E^{\perp}=\{0\} \times \R^{pn}$. Therefore $E^\perp$ can be identified with $\R^{pn}$ by simply forgetting the first $n$ coordinates. Under this identification, it is easy to see that 
\begin{align*}
    P_{E^\perp} \circ \Delta_{p + 1} &= \Delta_p \\
    P_{E^\perp} \circ \tilde\imath_k &= \imath_k, \quad k = 1, \dots, p.
\end{align*}
Hence, we obtain from \eqref{eq:proa_pplus1} that
\begin{equation}\label{eq:diag_mv}
\mathcal V_{a_0, \dots, a_p} = \binom{(p + 1) n}{n}^{-1} \vol_n(K) \cdot V_{pn}(-\Delta_p K[a_0], \imath_1 K[a_1], \dots, \imath_p K[a_p]),
\end{equation}
which, up to a binomial coefficient and the factor of $\vol_n(K)$, is precisely the left-hand-side of \eqref{eq:higher_godbersen_redux}.

Thus, it remains to bound $\mathcal V_{a_0, \dots, a_p}$ from above. Let 
\[K^{p + 1} = \overbrace{K \times K \times \cdots \times K}^{\text{$p + 1$ times}} = \tilde\imath_0 K + \cdots + \tilde \imath_p K.\]
Then clearly we have $\Delta_{p + 1} K \subset K^{p + 1}$, and so, returning to the original definition \eqref{eq:lifted_mv} of $\mathcal V_{a_0, \dots, a_p}$ and using monotonicity of mixed volumes, we have
\begin{align*}
    \mathcal V_{a_0, \dots, a_p} &= V_{(p + 1)n}(\Delta_{p + 1} K[n], \tilde\imath_0 K[a_0], \dots, \tilde\imath_p K[a_p]) \\
    &\le V_{(p + 1)n}(K^{p + 1}[n], \tilde\imath_0 K[a_0], \dots, \tilde\imath_p K[a_p]).
\end{align*}

We expand the right-hand side by multilinearity:
\begin{equation}\label{eq:mv_prod_expansion}
\mathcal V_{a_0, \dots, a_p} \le \underset{\substack{s_0, s_1, \dots, s_p \ge 0 \\ s_0 + \cdots + s_p = n}}\sum \binom{n}{s_0, \dots, s_p} V_{(p + 1)n}( \tilde\imath_0 K[a_0 + s_0], \tilde\imath_1 K[a_1 + s_1], \dots, \tilde\imath_p K[a_p + s_p]).
\end{equation}
For each $i = 0, \dots, p$, $\tilde\imath_j K$ is a convex set whose affine span has dimension $n$, so the only terms in the sum on the right-hand side of \eqref{eq:mv_prod_expansion} which can be non-zero are those for which $a_i + s_i \le n$ for all $i$, and hence satisfy 
$$\sum_{i = 0}^p s_i \leq \sum_{i=0}^p (n-a_i)=(p+1)n-pn=n.$$ 
Since $s_0 + \cdots + s_p =n$, we see that the only non-vanishing term corresponds to the case $s_i=n-a_i$ for all $i$. Thus we obtain

\begin{equation}\label{eq:mv_prod_reduction}
\mathcal V_{a_0, \dots, a_p} \le \binom{n}{n - a_0, \dots, n - a_p} V_{(p + 1)n}( \tilde\imath_0 K[n], \tilde\imath_1 K[n], \dots, \tilde\imath_p K[n]).
\end{equation}

On the right-hand side of \eqref{eq:mv_prod_reduction}, we have $p + 1$ bodies of dimension $n$ lying in mutually orthogonal subspaces, and so iteratively applying the projection formula for mixed volumes (Lemma \ref{lem:exact}) to each coordinate block, we see that the right-hand side equals
\[\binom{n}{n - a_0, \dots, n - a_p} \frac{(n!)^{p + 1}}{((p + 1)n)!} \vol_n(K)^{p + 1}.\]
Finally, combining the inequality
\[\mathcal V_{a_0, \dots, a_p} \le \binom{n}{n - a_0, \dots, n - a_p} \frac{(n!)^{p + 1}}{((p + 1)n)!} \vol_n(K)^{p + 1}\]
with the identity \eqref{eq:diag_mv}, we obtain
\begin{multline*}
\binom{n}{n - a_0, \dots, n - a_p}^{-1}  V_{pn}(-\Delta_p K[a_0], \imath_1 K[a_1], \dots, \imath_p K[a_p]) \\
\le \binom{(p + 1) n}{n} \vol_n(K)^{-1} \cdot \frac{(n!)^{p + 1}}{((p + 1)n)!} \vol_n(K)^{p + 1} = \frac{(n!)^p}{(pn)!} \vol_n(K)^p,
\end{multline*}
which is precisely \eqref{eq:higher_godbersen_redux}. The proof is complete.
\end{proof}

\section{Proof of Theorem \ref{thm:godbersen_equality} : equality case in Godbersen inequality}
\label{sec:godbersen_equality}
\label{section:godbersenclassique}

In this section, $K \in \mathcal K^n$ is a fixed convex body. In $\mathbb R^{2n}$, let $i_1, i_2, \Delta: \mathbb R^n \to \mathbb R^{2n}$ be the coordinate embeddings and diagonal map, respectively. For integers $a,b,c \geq 0$ with $a + b + c = 2n$, we denote $H_{a,b,c}= \binom{2n}{n} \mathcal V_{c,a,b}$, i.e.
\[H_{a,b,c}=\binom{2n}{n} V_{2n}(i_1 K[a],i_2 K[b], -\Delta K[c]).\]

The following is a special case of \cite[Remark 3.4]{K2025} (which says that $\mathcal V_{a_0, \dots, a_p}$ only depends on $\{a_0, \ldots, a_p\}$, not on their order); we prove it here for completeness.

\begin{lem}
    For any integers $a,b,c\geq 0$ summing to $2n$, one has $H_{a,b,c}=H_{b,a,c}=H_{a,c,b}$. In other words, $H_{a,b,c}$ only depends on $K$ and on $\{a,b,c\}$.
\end{lem}
\begin{proof} Consider the linear map $\pi_{1,2}$ given by $\pi_{1,2}(x,y)=(y,x)$. One has $\det \pi_{1, 2} = (-1)^n$, $\pi_{1, 2} \circ i_1 = i_2$, $\pi_{1, 2} \circ i_2 = i_1$ and $\pi_{1,2} \circ \Delta=\Delta$. So for any $a,b,c\ge 0$,
\begin{align*}
    V_{2n}(i_1 K[a],i_2 K[b],-\Delta K[c]) &=V_{2n}(\pi_{1,2}( i_1 K)[a],\pi_{1,2}(i_2 K)[b],\pi_{1,2}(-\Delta K)[c]) \\
    &=
    V_{2n}(i_2 K[a],i_1 K[b],-\Delta K[c]),
\end{align*}
i.e., $H_{a,b,c}=H_{b,a,c}$.

Similarly, let $T$ be the transformation with block matrix $\left(\begin{smallmatrix} I_n & -I_n \\ 0 & -I_n\end{smallmatrix}\right)$. We have $\det T = (-1)^n$, $T\circ i_1=i_1$, $T\circ i_2=-\Delta$ and $T\circ (-\Delta)=i_2$.
Therefore
\begin{align*}
    V_{2n}(i_1 K[a],i_2 K[b],-\Delta K[c]) &=V_{2n}(T(i_1 K)[a],T(i_2 K)[b],T(-\Delta K)[c]) \\
    &=V_{2n}(i_1K[a],-\Delta K[b],i_2 K[c]),
\end{align*}
and so $H_{a,b,c}=H_{a,c,b}$.
\end{proof}

Taking $p=2$ in Theorem \ref{thm:higher_godbersen}, we find that
\begin{equation}
\label{eqn:generalisedgodbersen}
 H_{a,b,c}\leq  \binom{n}{n-a,n-b,n-c} \vol_n(K)^2.
\end{equation}

\begin{prop}
\label{prop:moreequalities-AFargument}
Fix $1\leq k\leq n-1$. Assume $K\subset \R^n$ is a convex body satisfying 
\begin{equation}\label{eq:godbersen_equality_lifted}
V_{2n}(K^2[n],i_1 K[k],i_2 K[n-k])=V_{2n}(\Delta K[n],i_1 K[k],i_2 K[n-k]).
\end{equation}
Then for any $s \in \{1, \dots, n\}$ and any compact convex sets $L_1, \dots, L_{n-s} \in \mathcal K^{2n}$, one has
\begin{multline}\label{eq:godbersen_free}
V_{2n}(\Delta K[s], L_1,\dots, L_{n-s},i_1 K[k],i_2 K[n-k]) \\ 
= V_{2n}(K^2[s], L_1,\dots, L_{n-s},i_1 K[k],i_2 K[n-k]).
\end{multline}
\end{prop}

\begin{proof} To simplify the notation, let $\mathcal{C} = (i_1 K[k], i_2 K[n-k])$. We will prove by induction on $j \in \{0, \dots, n-1\}$ that for any compact convex sets $L_1, \dots, L_j \in \mathcal{K}^{2n}$, the mixed volume
\[M_r^{(j)}(L_1,\dots,L_j) := V_{2n}(L_1, \dots, L_j, \Delta K[n-j-r], K^2[r], \mathcal{C})\]
is independent of $r\in \{0,1, \dots, n-j\}$. In the proof, we may assume that $L_1, \ldots, L_j$ are full-dimensional, which ensures that the $M_r^{(j)}$ are strictly positive; the general case follows by approximation and by continuity of mixed volumes.

When $j=0$, we simply write $M_r^{(0)} = M_r$. Since $\Delta K \subset K^2$, monotonicity of mixed volumes yields
\[M_0 \leq M_1 \leq \dots \leq M_{n-1} \leq M_n.\]
Our assumption is that $M_0 = M_n$, which implies that $M_0 = M_1 = \dots = M_n$. This is the case $j=0$ of our claim.

Assume now that the claim holds for some integer $0 \leq j \leq n-2$. Let $L_1, \dots, L_j \in \mathcal{K}^{2n}$ be fixed, and write $M_r^{(j)}=V_{2n}(L_1, \dots,L_j,\Delta K[n-j-r], K^2[r],\mathcal{C})$. By the induction hypothesis, $M_0^{(j)} = \dots = M_{n-j}^{(j)}$. In particular, for any $1 \leq r \leq n-j-1$, we have equality in the Alexandrov-Fenchel inequality
$(M_r^{(j)})^2 \geq M_{r-1}^{(j)} M_{r+1}^{(j)}.$
By Theorem \ref{thm:af_equality}, this equality implies that there exists a constant $\alpha > 0$ such that for any convex body $M \in \mathcal{K}^{2n}$,
\begin{multline*}
    V_{2n}(M, L_1, \dots, L_j, \Delta K[n-j-r], K^2[r-1], \mathcal{C}) \\ 
    = \alpha V_{2n}(M, L_1, \dots, L_j, \Delta K[n-j-r-1], K^2[r], \mathcal{C}).
\end{multline*}
Choosing $M = K^2$ gives $M_r^{(j)} = \alpha M_{r+1}^{(j)}$. Since $M_r^{(j)} = M_{r+1}^{(j)}$, we deduce that $\alpha = 1$. Consequently, for any arbitrary compact convex set $L_{j+1} \in \mathcal{K}^{2n}$, taking $M = L_{j+1}$ in the above equation yields:
\[M_{r-1}^{(j+1)} (L_1,\dots,L_{j+1}) = M_r^{(j+1)}(L_1,\dots,L_{j+1})\]
for all $1 \leq r \leq n-j-1$. This proves that the sequence $(M_r^{(j+1)})_{0 \leq r \leq n-j-1}$ is constant, completing the induction step.
\end{proof}

Examining the proof of \cite[Theorem 1.1]{KM2026}, one sees that equality in Godbersen's inequality \eqref{eq:Godbersen-ineq} is equivalent to the assumption \eqref{eq:godbersen_equality_lifted} of Proposition \ref{prop:moreequalities-AFargument}, and hence, if $V_n(K[k], -K[n - k]) = \binom{n}{k} \vol_n(K)$ then $K$ satisfies \eqref{eq:godbersen_free}.

For any $s \in \{0, \dots, n\}$, let
\[\alpha_s = V_n(K[n - s], -K[s]).\]
Note that $\alpha_s = \alpha_{n - s}$ for any $s$. We also set $v = \alpha_0 = \vol(K)$.

\begin{cor}
    Let $1\leq k\leq n-1$. Assume $K\subset \R^n$ is a convex body satisfying 
    \[\alpha_k =V_n(K[n-k],-K[k])=\binom{n}{k} \vol_n(K).\]
    Then for any $s\in \{1,\dots,n-1\}$, one has
\begin{equation}
\label{eqn:akasidentity}
\alpha_s \alpha_k=\sum_{i=0}^s \binom{s}{i}  H_{k+i, n-k+s-i,n-s}.
\end{equation}
\end{cor}
\begin{proof}
Applying Proposition \ref{prop:moreequalities-AFargument} with $L_1=\dots=L_{n-s}=-\Delta K$ yields
 \begin{multline}
 \label{eqn:equalitywiths}
 \binom{2n}{n} V_{2n}(\Delta K[s], -\Delta K[n-s],i_1 K[k],i_2 K[n-k]) \\ 
= \binom{2n}{n} V_{2n}(K^2[s], -\Delta K[n-s],i_1 K[k],i_2 K[n-k]).
 \end{multline}
 
Let $E = \Delta \mathbb R^n$, and denote $\tilde{\Delta}(x)=(x,-x)$ so $\tilde{\Delta}(\R^n)=E^{\perp}$. By Lemma \ref{lem:exact}, the LHS rewrites
 \begin{align}\label{eq:equalitywiths_lhs}
 & \binom{2n}{n} V_{2n}(\Delta K[s], -\Delta K[n-s],i_1 K[k],i_2 K[n-k]) \nonumber \\
 &\qquad = V_E(\Delta K[s],-\Delta K[n-s]) V_{E^{\perp}}((P_{E^\perp}\circ i_1)(K)[k],(P_{E^\perp}\circ i_2)(K)[n-k]) \nonumber \\
 &\qquad =2^{-n}V_E(\Delta K[s],-\Delta K[n-s])V_{E^{\perp}}(\tilde{\Delta}(K)[k],\tilde{\Delta}(-K)[n-k]) \nonumber  \\
 &\qquad = \alpha_s \alpha_k,
 \end{align}
 where we used that $P_{E^\perp} \circ i_1=\frac{1}{2}\tilde{\Delta}$, $P_{E^\perp} \circ i_2=-\frac{1}{2}\tilde{\Delta}$ along with homogeneity of mixed volumes for the second equality, and that $|\det(\Delta)|=|\det(\tilde{\Delta})|=2^{n/2}$ for the last equality.
By multilinearity, since $K^2=i_1K+i_2 K$, the RHS of \eqref{eqn:equalitywiths} rewrites as
 \begin{equation} 
 \label{eq:equalitywiths_rhs}
 \binom{2n}{n}V_{2n}(K^2[s],-\Delta K[n-s],i_1K[k],i_2K[n-k])=\sum_{i=0}^s \binom{s}{i} H_{k+i,n-k+s-i, n-s}
 \end{equation}
 where we recall the notation $H_{a,b,c}=\binom{2n}{n}V_{2n}(i_1 K[a],i_2 K[b], -\Delta K[c])$. Combining \eqref{eqn:equalitywiths}, \eqref{eq:equalitywiths_lhs}, and \eqref{eq:equalitywiths_rhs} gives \eqref{eqn:akasidentity}.
\end{proof}

We are ready for the proof of Theorem \ref{thm:godbersen_equality}.

\begin{proof}[Proof of Theorem \ref{thm:godbersen_equality}]
Let $1\leq k\leq n-1$, and assume that $K\in \mathcal{K}^n$ is a convex body such that $V_n(K[n-k],-K[k])=  \binom{n}{k} \vol_n(K)$, i.e., $\alpha_k=\binom{n}{k} v$. To show that $K$ must be a simplex, we will show that $\alpha_1=n v$.

Taking $s=k$ in \eqref{eqn:akasidentity}, and using \eqref{eqn:generalisedgodbersen} with $(a,b,c)=(k+i,n-i,n-k)$, $i\in \{0,\dots, k\}$, one finds
\begin{align*} 
\alpha_k^2=\sum_{i=0}^k  \binom{k}{i} H_{k+i,n-i,n-k} &\leq \sum_{i=0}^k \binom{k}{k - i} \binom{n}{n-k-i,i,k} v^2 \\ 
&=\binom{n}{k} \sum_{i=0}^k \binom{k}{k-i}\binom{n-k}{i} v = \binom{n}{k}^2 v^2= \alpha_k^2,
\end{align*}
where the last equality is due to our assumption $\alpha_k = \binom{n}{k} v$. Therefore we must have $H_{a,b,c}= \binom{n}{n-a,n-b,n-c} v^2$ for any $\{a,b,c\}\in \{\{k+i,n-i,n-k\}; 0\leq i\leq k\}$. In particular, taking $i=1$ yields 
\begin{equation}
    \label{eqn:eli}
    H_{k+1,n-k,n-1}=\binom{n}{n-k-1,k,1} v^2 =(n-k)\binom{n}{k} v^2.
\end{equation}

Similarly, applying \eqref{eqn:akasidentity} with $k$ replaced by $n - k$ (which is valid since $\alpha_k = \alpha_{n - k}$), taking $s = n - k$, and using \eqref{eqn:generalisedgodbersen}, one finds
\begin{align*}\alpha_{n-k}^2 = \sum_{i=0}^{n-k} \binom{n-k}{i} H_{n-k+i,n-i,k} &\leq \sum_{i=0}^{n-k} \binom{n-k}{i} \binom{n}{k-i,i,n-k} v^2 \\
&= \binom{n}{k} \sum_{i=0}^{n-k} \binom{n-k}{i}\binom{k}{k-i} v^2=\binom{n}{k}^2 v^2 = \alpha_{n-k}^2.
\end{align*}
Thus the inequalities $H_{a,b,c}\leq \binom{n}{n-a,n-b,n-c} v^2$ must be equalities for 
\[\{a,b,c\}\in \{\{n-k+i,k,n-i\} ; 0\leq i\leq n-k\}.\]
In particular, choosing $i=1$ we get
\begin{equation}
\label{eqn:maud}
H_{k,n-k+1,n-1} = \binom{n}{n-k, k-1, 1} v^2 = k\binom{n}{k} v^2.
\end{equation}

Now, taking $s=1$ in \eqref{eqn:akasidentity} yields
\begin{align*}
    \alpha_1 \cdot \alpha_k  &=H_{k,n-k+1,n-1}+H_{k+1,n-k,n-1} \\
    &=k \binom{n}{k}v^2 +(n-k)\binom{n}{k} v^2 \\
    &= n \binom{n}{k} v^2 = n v\cdot \alpha_k,
\end{align*}
where we used \eqref{eqn:eli} and \eqref{eqn:maud} for the second equality, and the assumption that $\alpha_k=\binom{n}{k} v$ for the last equality. Dividing through by $\alpha_k>0$, we obtain $\alpha_1 = n v$, i.e., $V_n(K[n - 1], -K) = n \vol_n(K)$, from which we deduce by Proposition \ref{prop:godbersen_j=1_equality} that $K$ is a simplex.
\end{proof}

\section{Equality in the higher order Godbersen inequality}\label{sec:higher_godbersen_equality}

To characterize the equality case in Godbersen's inequality \eqref{eq:Godbersen-ineq}, which corresponds to the case $p=1$ of (the equality case in) the   inequality \eqref{eq:higher_godbersen}, we passed to $\R^{2n}$ and used inequalities for $H_{a, b, c}$ (namely the case $p = 2$ of Theorem \ref{thm:higher_godbersen}).
In this section, we will use a similar, though slightly more complicated, scheme of proof to prove Theorem \ref{thm:higher_godbersen_equality}. 
 
Recall that the proof of the higher-order Godbersen inequality (Theorem \ref{thm:higher_godbersen}), which involves $V_{pn}(-\Delta_p K[a_0], \ldots,i_p K[a_p])$, uses a lifting of this mixed volume to $\R^{(p+1)n}$ and a monotonicity inequality coming from $\Delta_p K\subset K^{p+1}$. To prove the equality case of Conjecture \ref{conj:higher_godbersen}, we shall use a monotonicity inequality coming from a similar inclusion, in $\R^{2pn}$, and use the inequalities of Theorem \ref{thm:higher_godbersen} with $p'=2p$ to show that equality in \eqref{eq:higher_godbersen_redux} forces equality of certain order-$2p$ mixed volumes. We then use these inequalities, along with  Proposition \ref{prop:moreequalities-AFargument}, to conclude, just as in \S \ref{sec:godbersen_equality}, that equality holds in Godbersen's inequality $V(-K, K[n - 1])\leq n \vol_n(K)$ as well.

\vspace{2mm}
For any $q\geq 1$, and any integers $c_0, \ldots,c_q\geq 0$ such that $c_0 + \ldots +c_q=qn$, let us denote 
\begin{align*}
    v^{(q)}_{c_0, \ldots,c_q}
    &= \binom{qn}{n,\ldots,n} V_{qn}(-\Delta_q K[c_0], i_1 K[c_1], \ldots,i_q K[c_q]) \\
    &=\frac{(qn)!}{(n!)^q}V_{qn}(-\Delta_q K[c_0], i_1 K[c_1], \ldots,i_q K[c_q]).
\end{align*}
For instance, in the notation of the preceding section, taking $q=1$, $c_0=k$ and $c_1 = n - k$ yields \[v^{(1)}_{k,n-k}=V_n(K[n-k],-K[k]) = \alpha_k,\]
while taking $q=2$ and $a,b,c\geq 0$ summing to $2n$ yields $v^{(2)}_{a,b,c} = H_{a,b,c}$. 

By our choice of normalization for the $v^{(q)}_{c_0, \ldots,c_q}$, it follows immediately from Lemma \ref{lem:exact} that for any $q\geq 2$, $1\leq k\leq q$, and any $c_i\geq 0$ such that $c_{k+1}= \ldots =c_q=n$ and $c_0+ \ldots +c_k=kn$, one has
\begin{equation}
\label{eq:dimensiondown}
v^{(q)}_{c_0, \ldots,c_q}=v^{(k)}_{c_0, \ldots,c_k} \cdot v^{q-k},
\end{equation}
where we write $v=\vol_n(K)$ as before.

In this notation, the statement of Theorem \ref{thm:higher_godbersen_equality} is as follows: if $K \subset \mathbb R^n$ is a convex body, $p \ge 1$, and $a_0, \ldots, a_p \ge 1$ are such that 
\begin{equation}\label{eq:higher_godbersen_normalized}
v^{(p)}_{a_0, \ldots, a_p} = \binom{n}{n - a_0, \ldots, n - a_p} v^p,
\end{equation}
then $K$ is an $n$-simplex. In proving this, we may assume that $1\leq a_0, \ldots,a_p \leq n-1$. Otherwise, say for instance $a_p=n$, one has 
\[v^{(p)}_{a_0, \ldots,a_{p-1},n}=v^{(p-1)}_{a_0, \ldots,a_{p-1}} \cdot v \quad \text{ and} \qquad \binom{n}{n-a_0, \ldots,n-a_p} = \binom{n}{n-a_0, \ldots,n-a_{p-1}},\]
so the assumption 
\[v^{(p)}_{a_0, \ldots,a_p}=\binom{n}{n-a_0, \ldots,n-a_p} \cdot v^p\]
reduces to
\[v^{(p-1)}_{a_0, \ldots,a_{p-1}}=\binom{n}{n - a_0, \ldots, n - a_{p - 1}} \cdot v^{p-1},\] which by induction on $p$ implies that $K$ must be an $n$-simplex. 

We now begin the proof that \eqref{eq:higher_godbersen_normalized} implies that $K$ is an $n$-simplex. Let $\Delta_p, \Delta_{2p}$ be the diagonal embeddings of $\mathbb R^n$ into $\mathbb R^{pn}$ and $\mathbb R^{2pn}$, respectively. Let $i_1, \ldots,i_p,j_1, \ldots, j_p$ denote the $2p$ block coordinate embeddings of $\R^{n}$ into $\R^{2pn}$, (i.e., the image of $j_\ell$ consists of vectors supported on the $(p+\ell)$-th coordinate block). Let 
\begin{align*}
E_0 &=i_1 \R^n + \ldots +i_p \R^n=\R^{pn}\times \{0\} \\
E_1 &= j_1 \R^n + \cdots + j_p \R^n = \{0\}\times \R^{pn}=E_0^{\perp}
\end{align*} 
denote the two coordinate copies of $\mathbb R^{pn}$ embedded into $\mathbb R^{2pn}$, and let 
\[\bar{j}_i: \mathbb R^{pn} \to E_i \subset \mathbb R^{2pn},\quad i = 0, 1,\]
be the corresponding embeddings. Finally, denote 
$\Delta_{p, i}(x) = \bar{j}_i \circ \Delta_p: \mathbb R^n \to \mathbb R^{2pn},$
i.e.
\begin{align*}
    \Delta_{p,0}(x) &= (x, \ldots,x,0, \ldots,0) \quad \text{ and } \quad 
    \Delta_{p,1}(x)= (0, \ldots,0,x, \ldots,x).
\end{align*} 
Note that $\Delta_{2p}=\Delta_{p,0}+\Delta_{p,1}$ and so $P_{E_1}\circ \Delta_{2p}=\Delta_{p,1}$.

Observe that
\begin{equation}
\label{eq:inclusionp0}
-\Delta_{p,0} K\subset -\Delta_{2p} K +j_1 K+ \ldots +j_p K =: L.
\end{equation}

Therefore, by monotonicity of mixed volumes, we have
\begin{multline}
\label{eq:step0higherequality}
V_{2pn}(-\Delta_{p,0} K[a_0], i_1 K[a_1], \ldots,i_p K[a_p], -\Delta_{2p} K[a_0], j_1 K[a_1],  \ldots  , j_p K[a_p]) \\ 
\leq V_{2pn}(L[a_0], i_1 K[a_1], \ldots,i_p K[a_p], -\Delta_{2p} K[a_0], j_1 K[a_1],  \ldots  , j_p K[a_p]).
\end{multline}
Note that $i_\ell K \subset E_0$ for $1\leq \ell \leq p$ and that $\Delta_{p,0} K \subset E_0$. Using Lemma \ref{lem:exact} with $N=2pn, d=pn$ and $(L_1, \ldots,L_d)=(-\Delta_{p,0} K[a_0], i_1 K[a_1], \ldots,i_p K[a_p])$, and recalling that $P_{E_1} \Delta_{2p}=\Delta_{p,1}$, we may rewrite the left-hand side of \eqref{eq:step0higherequality} as follows:
\begin{multline}
V_{2pn}(-\Delta_{p,0} K[a_0], i_1 K[a_1], \ldots,i_p K[a_p], -\Delta_{2p} K[a_0], j_1 K[a_1],  \ldots  , j_p K[a_p]) \\ = \frac{(pn)!^2}{(2pn)!} \frac{(n!)^{2p}}{(pn)!^2}\left(v^{(p)}_{a_0, \ldots,a_p}\right)^2 =\binom{2pn}{n, \ldots,n}^{-1} \left(v^{(p)}_{a_0, \ldots,a_p}\right)^2.
\end{multline}
On the other hand, since $L =-\Delta_{2p} K+j_1 K+ \ldots +j_p K$, by multilinearity we can rewrite the right-hand side of  \eqref{eq:step0higherequality} as:
\begin{align*}
&V_{2pn}(L[a_0], i_1 K[a_1], \ldots,i_p K[a_p], -\Delta_{2p} K[a_0], j_1 K[a_1],  \ldots  , j_p K[a_p]) \\
&\quad= \sum_{\substack{s_0+ \ldots +s_p=a_0 \\ s_0, \ldots,s_p \ge 0}} \binom{a_0}{s_0, \ldots,s_p} V_{2pn}(-\Delta_{2p} K[a_0+s_0],j_1 K[a_1+s_1], \ldots, j_p K[a_p+s_p], i_1 K[a_1], \ldots,i_p K[a_p]) \\
&\quad= \binom{2pn}{n, \ldots,n}^{-1} \sum_{\substack{s_0+ \ldots +s_p=a_0 \\ s_0, \ldots,s_p \ge 0}}  \binom{a_0}{s_0, \ldots,s_p} v^{(2p)}_{a_0+s_0, \ldots,a_p+s_p,a_1, \ldots,a_p}.
\end{align*}
Rewriting \eqref{eq:step0higherequality} using the above and applying the higher-order Godbersen inequality \eqref{eq:higher_godbersen}, we obtain
\begin{align*}
\left(v^{(p)}_{a_0, \ldots,a_p}\right)^2 & \leq \sum_{\substack{s_0+ \ldots +s_p=a_0 \\ s_0, \ldots,s_p \ge 0}}  \binom{a_0}{s_0, \ldots,s_p} v^{(2p)}_{a_0+s_0, \ldots,a_p+s_p,a_1, \ldots,a_p} \\
& \leq \sum_{\substack{s_0+ \ldots +s_p=a_0 \\ s_0, \ldots,s_p \ge 0}}  \binom{a_0}{s_0, \ldots,s_p} \binom{n}{n-a_1, \ldots,n-a_p,n-a_0-s_0, \ldots,n-a_p-s_p} v^{2p} \\
& = \binom{n}{a_0} \binom{a_0}{n-a_1, \ldots,n-a_p}\sum_{\substack{s_0+ \ldots +s_p=a_0 \\ s_0, \ldots,s_p \ge 0}}  \binom{a_0}{s_0, \ldots,s_p}
 \binom{n-a_0}{n-a_0-s_0, \ldots,n-a_p-s_p} v^{2p} \\
 &= \binom{n}{a_0} \binom{a_0}{n-a_1, \ldots,n-a_p} \binom{n}{n-a_0, \ldots,n-a_p} v^{2p} =\left( \binom{n}{n-a_0, \ldots,n-a_p} v^p \right)^2,
\end{align*}
where the first equality comes from the identity
$$\binom{n}{k_1, \ldots,k_p, b_1, \ldots,b_m}=\binom{n}{k}\binom{k}{k_1, \ldots,k_p}\binom{n-k}{b_1, \ldots,b_m},$$
for $k_1+ \ldots +k_p=k$ and $b_1+ \ldots +b_m=n-k$, and the second equality is a multinomial Vandermonde identity.

Since we are assuming that $v^{(p)}_{a_0, \ldots,a_p}=\binom{n}{n-a_0, \ldots,n-a_p} v^p$, we deduce that for any $s_0, \ldots,s_p\geq 0$ such that $s_0+ \ldots +s_p=a_0$, we have 
\begin{equation}
\label{eq:tobeused0}
v^{(2p)}_{a_0+s_0, \ldots,a_p+s_p,a_1, \ldots,a_p}= \binom{n}{n-a_1, \ldots,n-a_p,n-a_0-s_0, \ldots,n-a_p-s_p} v^{2p}.
\end{equation} 
By symmetry in $(a_0, \ldots,a_p)$, we also get that, for any $1\leq \ell \leq p$,
\begin{multline}
\label{eq:tobeusedell}
v^{(2p)}_{a_0+s_0, \ldots,a_p+s_p,a_0, \ldots,a_{\ell-1},a_{\ell+1},\ldots, a_p} \\
= \binom{n}{n-a_0, \ldots n-a_{\ell-1},n-a_{\ell+1},  \ldots  ,n-a_p,n-a_0-s_0, \ldots,n-a_p-s_p} v^{2p}
\end{multline}
whenever $s_0, \ldots, s_p\geq 0$ have sum $s_0+ \ldots +s_p=a_{\ell}$.

From here on, we adopt the notations of the proof of Theorem \ref{thm:higher_godbersen}: let $\tilde\imath_0, \ldots, \tilde\imath_p$ denote the $(p+1)$ block coordinate embeddings from $\R^n$ to $\R^{(p+1)n}$, and $\Delta_{p + 1}$ denotes the diagonal embedding. Recall from the proof of Theorem \ref{thm:higher_godbersen} that the assumption $v^{(p)}_{a_0, \ldots,a_p}=\binom{n}{n-a_0, \ldots,n-a_p} v^p$ is equivalent to assuming 
\[V_{(p+1)n}(\Delta_{p + 1} K[n], \tilde\imath_0 K[a_0], \ldots,\tilde\imath_pK[a_p])=V_{(p+1)n}(K^{p+1}[n], \tilde\imath_0 K[a_0], \ldots,\tilde\imath_pK[a_p]).\]
By Proposition \ref{prop:moreequalities-AFargument} , this implies that for each $s\in \{1,  \ldots, n\}$ and any $L_1, \ldots, L_{n-s}\in \mathcal{K}^{(p+1)n}$, one has
\begin{multline}
\label{eqn:afterAFtrick}
V_{(p+1)n}(\Delta_{p + 1} K[s], L_1, \ldots,L_{n-s}, \tilde\imath_0 K[a_0], \ldots,\tilde\imath_pK[a_p]) \\ 
= V_{(p+1)n}(K^{p+1}[s],L_1, \ldots,L_{n-s}, \tilde\imath_0 K[a_0], \ldots,\tilde\imath_pK[a_p]).
\end{multline}
In particular, taking $s=1$ and $L_j=-\Delta_{p + 1} K$ in \eqref{eqn:afterAFtrick}, our assumption implies that 
\begin{multline}
\label{eq:csqcewith1}
\frac{((p+1)n)!}{(n!)^{p+1}}V_{(p+1)n}(\Delta_{p + 1} K, -\Delta_{p + 1} K[n-1], \tilde\imath_0 K[a_0], \ldots,\tilde\imath_pK[a_p]) \\ 
= \frac{((p+1)n)!}{(n!)^{p+1}}V_{(p+1)n}(K^{p+1},-\Delta K[n-1], \tilde\imath_0 K[a_0], \ldots,\tilde\imath_pK[a_p]).
\end{multline}
By multilinearity (noting that $(n - 1) + a_0 + \cdots + a_p = (p + 1)n - 1$), the right-hand side of \eqref{eq:csqcewith1} expands as
\[\text{RHS} =\sum_{\ell=0}^p v^{(p+1)}_{n-1,a_0, \ldots, a_{\ell-1},a_{\ell}+1, a_{\ell+1}, \ldots,a_p},\]
while by the projection identity \eqref{eq:projectionidentity}, the left-hand side of \eqref{eq:csqcewith1} equals
\[\text{LHS}=\frac{((p+1)n)!}{(n!)^{p+1}} \binom{(p+1)n}{n}^{-1}\alpha_1 \cdot \frac{(n!)^p}{(pn)!} v^{(p)}_{a_0, \ldots,a_p}=\alpha_1 \cdot v^{(p)}_{a_0, \ldots,a_p}, \]
where we recall the notation $\alpha_1=V_n(K,-K[n-1])$.

Substituting in \eqref{eq:csqcewith1}, we thus obtain
\begin{equation}
    \label{eq:yoel23}
    \alpha_1 \cdot v^{(p)}_{a_0, \ldots,a_p}=\sum_{\ell=0}^p v^{(p+1)}_{n-1,a_0, \ldots, a_{\ell-1},a_{\ell}+1, \ldots,a_p}.
\end{equation}
 
Recall that we are assuming $a_{\ell} <n$ for $0\leq \ell \leq p$. Thus, we may apply \eqref{eq:tobeused0} with $s_0 = 1, s_1 = n-a_1 -1$, and $s_j=n-a_j$ for $2\leq j\leq p$ to obtain
 \begin{equation}
 \label{eq:aaron0}
v^{(2p)}_{a_0+1,n-1, n[p-1],a_1, \ldots,a_p}= \binom{n}{n-a_1, \ldots,n-a_p,n-a_0-1,1} v^{2p},
 \end{equation}
 which by \eqref{eq:dimensiondown} simplifies to
 \begin{equation}
 \label{eq:meir0}
 v^{(p+1)}_{n-1,a_0+1,a_1, \ldots,a_p}=n \binom{n-1}{n-a_0-1,n-a_1, \ldots,n-a_p} \cdot v^{p+1}.
 \end{equation}
Similarly, for any $\ell \in \{1, \ldots, p\}$, applying \eqref{eq:tobeusedell} with $s_\ell=1$, $s_0=n-a_0-1$ and $s_j=n-a_j$ for $j\in [p]\setminus \{\ell\}$ yields that
\begin{equation}
 \label{eq:meirell}
 v^{(p+1)}_{n-1,a_0,a_1, \ldots,a_{\ell}+1, \ldots,a_p}=n \binom{n-1}{n-a_0,n-a_1, \ldots,n-a_{\ell}-1,  \ldots, n-a_p} \cdot v^{p+1}.
 \end{equation}
 Denote $C_{\ell}= \binom{n-1}{n-a_0,n-a_1, \ldots,n-a_{\ell}-1,  \ldots, n-a_p}$. Combining \eqref{eq:yoel23}, \eqref{eq:meir0} and \eqref{eq:meirell}, one finds
 \begin{align*}
\alpha_1 \cdot v^{(p)}_{a_0, \ldots,a_p}= n\cdot v^{p+1} \cdot \sum_{\ell=0}^p C_{\ell} =n \cdot v^{p+1} \cdot \binom{n}{n-a_0, \ldots, n-a_p},
\end{align*}
where the second equality uses the multinomial version of Pascal's identity. Using one last time our assumption that $v^{(p)}_{a_0, \ldots,a_p}= \binom{n}{n-a_0,  \ldots, n-a_p} v^p$, we find that $\alpha_1 = n \vol_n(K)$. By Proposition \ref{prop:godbersen_j=1_equality}, this means that $K$ must be an $n$-simplex, completing the proof.

\printbibliography

\end{document}